\documentclass[11pt,a4paper]{amsart}
\usepackage[T1]{fontenc}
\usepackage{lmodern}
\usepackage[margin=27mm]{geometry}
\usepackage{amsmath,amssymb,amsthm,mathtools}
\usepackage{mathrsfs}
\usepackage{booktabs,array}
\usepackage{microtype}
\usepackage{enumitem}
\usepackage{xcolor}
\usepackage[colorlinks=true,linkcolor=blue!50!black,citecolor=blue!50!black,urlcolor=blue!50!black]{hyperref}
\hypersetup{pdftitle={Local rigidity for the sharp Nash inequality},pdfauthor={}}
\numberwithin{equation}{section}
\newtheorem{theorem}{Theorem}[section]
\newtheorem{lemma}[theorem]{Lemma}
\newtheorem{proposition}[theorem]{Proposition}
\newtheorem{corollary}[theorem]{Corollary}
\theoremstyle{definition}
\newtheorem{definition}[theorem]{Definition}
\theoremstyle{remark}
\newtheorem{remark}[theorem]{Remark}
\DeclareMathOperator{\Ric}{Ric}
\DeclareMathOperator{\Rm}{Rm}
\DeclareMathOperator{\Sc}{Sc}

\newcommand{\R}{\mathbb R}
\newcommand{\dd}{\,dr}

\allowdisplaybreaks[2]
\setlist[enumerate]{itemsep=3pt,topsep=5pt}

\usepackage[nameinlink,noabbrev]{cleveref}
\crefname{lemma}{Lemma}{Lemmas}
\crefname{proposition}{Proposition}{Propositions}
\crefname{corollary}{Corollary}{Corollaries}
\crefname{theorem}{Theorem}{Theorems}
\crefname{remark}{Remark}{Remarks}
\crefname{appendix}{Appendix}{Appendices}
\newcommand{\CN}{C_{\mathrm N}}
\newcommand{\QN}{Q_{\mathrm N}}
\newcommand{\E}{\mathcal E}
\newcommand{\Dn}{\mathcal D_{n,2}}
\newcommand{\Hn}{\mathcal H_{n,2}}
\newcommand{\Ln}{\mathcal L_{n,2}}
\newcommand{\Gf}{\mathcal G_n}

\title{Local rigidity for the sharp Nash inequality}
\author{Liang Cheng, Zhenyu Lu$^{\ast}$ }
    \thanks{*Corresponding author} 
     

     \thanks{Liang Cheng's  Research partially supported by
          Natural Science Foundation of China 12171180
     }
     
     \address{School of Mathematics and Statistics, and Key Laboratory of Nonlinear Analysis $\&$ Applications (Ministry of Education), Central  China Normal University, Wuhan, 430079, P.R.China}

     \email{chengliang@ccnu.edu.cn }
     
     \address{School of Mathematics and Statistics, and Key Laboratory of Nonlinear Analysis $\&$ Applications (Ministry of Education), Central  China Normal University, Wuhan, 430079, P.R.China}
          \email{   quermassintegral@gmail.com}

\begin{document}
\maketitle
\begin{abstract}
Let $V$ be a nonempty relatively compact open subset of a smooth
Riemannian manifold $(M^n,g)$ without boundary. We prove that
$\CN(V,g)\ge\CN(\R^n)$ and $\int_V\Sc_g\,d\mu_g\ge0$ imply
$\mathrm{Rm}_g=0$ on $V$, where $\CN$ denotes the optimal constant in
the Nash inequality and $\Sc_g$ denotes the scalar curvature.
This improves an earlier rigidity result of  Druet, Hebey, and Vaugon [Int. Math. Res. Not. (1999), no. 14, 735--779].


\end{abstract}

\section{Introduction}\label{sec:intro}
\subsection{The main result}

Let $(M^n,g)$ be a Riemannian manifold and let
$V\subset M$. The Nash inequality
takes the form
\begin{equation}\label{eq:nash-inequality}
C\left(\int_V u^2\,d\mu_g\right)^{1+2/n}
\le
\left(\int_V |u|\,d\mu_g\right)^{4/n}
\int_V |\nabla u|_g^2\,d\mu_g,
\end{equation}
for every $u\in H_0^1(V)$, where $C>0$ is a constant.
Originally introduced by Nash in his study of the
regularity of solutions to parabolic and elliptic
equations \cite{Nash1958}, the Nash inequality has
become a fundamental tool in analysis, partial
differential equations, and geometric analysis.
It plays an important role in establishing smoothing
and decay estimates for heat semigroups and obtaining
heat kernel bounds \cite{Davies1989,SaloffCoste2002}.
It also reveals connections among Sobolev-type
inequalities, spectral theory, diffusion processes,
and the underlying geometry of Riemannian manifolds.

The optimal constant for the Nash inequality also carries
important geometric information and can lead to rigidity
results for Riemannian manifolds. 
Here, we use the following definition.
\begin{definition}
Let $(M^n,g)$ be a smooth Riemannian manifold. Suppose that $V\Subset M$ is a nonempty open set whose closure is compact in $M$.
    For every nonzero $u\in H_0^1(V)$, define the Nash quotient
\begin{equation}\label{eq:quotient}
\QN^g(u)=
\frac{\left(\int_V|u|\,d\mu_g\right)^{4/n}
\int_V|\nabla u|_g^2\,d\mu_g}
{\left(\int_Vu^2\,d\mu_g\right)^{1+2/n}}
=\frac{\|u\|_1^{4/n}\|\nabla u\|_2^2}{\|u\|_2^{2+4/n}},
\end{equation}
and the optimal Nash constant of $V$ is defined as 
\[
\CN(V,g)=\inf_{0\ne u\in H_0^1(V)}\QN^g(u).
\]
Here,  $d\mu_g$ denotes the Riemannian volume measure.
\end{definition}

Druet,
Hebey, and Vaugon \cite{DruetHebeyVaugon1999} proved
that a  Riemannian manifold satisfies the sharp
Euclidean Nash inequality and
\begin{equation}\label{eq:ric-nonnegative}
    \operatorname{Ric}_g\ge 0
\end{equation}
must be flat.
Motivated by their result, we show that the sharp Euclidean Nash inequality on an open set, together with the nonnegativity of its total scalar curvature, implies local flatness, thereby replacing the pointwise Ricci curvature condition \eqref{eq:ric-nonnegative} with a weaker integral condition. Precisely, the following theorem is the main
result of this paper.

\begin{theorem}\label{thm:main}
Let $(M^n,g)$ be a smooth Riemannian manifold without boundary, with $n\ge2$, and let $V\Subset M$ be a nonempty open set. Suppose that
\begin{equation}\label{eq:hypotheses}
\CN(V,g)\ge\CN(\R^n),\qquad \int_V S\,d\mu_g\ge0.
\end{equation}
Then $\Rm_g=0$ throughout $V$. Here, $\CN(\R^n)$ denotes the optimal Nash constant of $n$-dimensional Euclidean space and $S=\Sc_g$ denotes the scalar curvature of $g$.
\end{theorem}
\begin{remark}
Indeed, for any open nonempty subset \(V \subset M\), we always have
$
\CN(V,g)\le\CN(\R^n),
$
because the Nash quotient relative to a family of its extremal functions
\(U_{p,t}\) (see the definition in \eqref{eq:localized-U})
satisfies
$
\frac{\QN^g(U_{p,t})}{\CN(\R^n)}
\to 1,
$
see Theorem~1.3 in \cite{DruetHebeyVaugon1999}.
Hence, the hypothesis
$
\CN(V,g)\ge\CN(\R^n)
$
in Theorem~\ref{thm:main} implies
$
\CN(V,g)=\CN(\R^n).
$
\end{remark}

\begin{remark}
          The rigidity properties of Riemannian manifolds with lower scalar curvature bounds are an important subject of intensive study. 
          As highlighted by Gromov in the following problems: \textit{ Find verifiable criteria for extremality and rigidity, decide which manifolds admit extremal/rigid metrics and describe particular extremal/rigid manifolds; see Problem C in \cite{Gdozen}.}
          For a comprehensive overview of the rigidity results concerning scalar curvature, we refer to the survey \cite{Brendle} and the lectures \cite{Gromovlecture}, along with additional references therein.     Note that  Theorem \ref{thm:main} demonstrates rigidity with respect to the scalar curvature and the optimal Nash constant.
     \end{remark}

\subsection{Strategy of the proof}\label{sec:proof_strategy}

We outline the strategy of our proofs. By evaluating the Nash
quotient on rescaled Euclidean extremal functions concentrating
at a point, Druet, Hebey, and Vaugon \cite{DruetHebeyVaugon1999}
computed the first nontrivial term in its asymptotic expansion
in the concentration scale, showing that $S\le 0$ on $V$ if
$\CN(V,g)\ge\CN(\R^n)$. Combined with the assumption
$\int_V S\,d\mu_g\ge0$, this implies $S\equiv0$ on $V$.
However, this approach, when carried only to the second
nontrivial term, may still obscure finer local geometric
details, thereby limiting its utility in establishing the
desired rigidity results;  we elaborate on this point in Remark \ref{rem:ricci-coefficient}.

Our work is motivated by the first author~\cite{ChengLS},
who used a refined asymptotic expansion to analyze the
short-time behavior of the logarithmic Sobolev functional.
The construction modifies its standard Euclidean extremal,
$(4\pi t)^{-\frac{n}{4}} e^{-\frac{d(p,x)^2}{8t}}$, by
incorporating curvature-dependent correction terms. The test
function is thus chosen as
\[
u(x,t) = (4\pi t)^{-\frac{n}{4}} e^{-\frac{d(p,x)^2}{8t}} \eta(x,t).
\]
Here $\eta$ is a correction factor chosen so that, in normal
coordinates centered at $p$,
\begin{equation}\label{expansion_eta_ricci}
     \eta(x,0)^2 = 1 + A_{ij} x^i x^j + o(|x|^2),
\end{equation}
with $p\in\operatorname{supp}(\eta) \subset\subset V$, where
$A_{ij}=\tfrac{1}{3}\operatorname{Ric}_{ij}(p)$.
With this choice of $\eta$, it was shown in~\cite{ChengLS}
that the functionals admit the following clean short-time
asymptotic expansions:
\begin{equation}\label{expansion_nu_opt}
     \mathcal W(V,g,u,t) = -\frac{1}{6}|\operatorname{Rm}|^2(p)\,t^2 + o(t^2),
\end{equation}
and
\begin{equation}\label{expansion_n_opt}
     \mathcal L(V,g,u,t) = -\operatorname{Sc}(p)\,t - \left( \Delta \operatorname{Sc}(p) + \frac{1}{6}|\operatorname{Rm}|^2(p) \right) t^2 + o(t^2).
\end{equation}
These expansions were instrumental in establishing the following
rigidity result: if an open subset $V \subset M$ satisfies
$\int_V \operatorname{Sc} \, d\mu \geq 0$ and the optimal
constant in the logarithmic Sobolev inequality on $V$ is no
smaller than its Euclidean counterpart, then $V$ must be flat.
The same idea also appears in the study of a similar problem with respect to the Del
Pino--Dolbeault family of Gagliardo--Nirenberg inequalities in
the work of the first author~\cite{ChengGN}.

In this paper, our strategy follows an approach similar to the
refined asymptotic expansion techniques used in the previous works
\cite{ChengLS} and \cite{ChengGN}. However, we need to add a conformally adjusted
radial function $f$ to the correction term. Precisely, we consider the test function
\begin{equation}\label{eq:strategy-test}
 w_{t,f,A}(x)=t^{-n/4}
 \left[U\!\left(\frac{|x|}{\sqrt t}\right)+ f\!\left(\frac{|x|}{\sqrt t}\right)A_{ij}x^i x^j\right],
 \qquad |x|<\sqrt t,
\end{equation}
where $t^{-n/4}U(|x|/\sqrt t)$ is the extremal function for the
Nash quotient (see \eqref{eq:localized-U}), and $\{x^i\}$ are
normal coordinates centered at some point $p$. Here $A$ is a
trace-free symmetric tensor at $p$, and the smooth function $f$
will be chosen appropriately. We remark that without $f$ the construction does not suffice to establish the desired rigidity; this is explained in Remark \ref{rem:ricci-coefficient}.

We calculate the expansion of $\QN^g(w_{t,f,A})$
up to order $t^2$ (see \cref{prop:corrected}):
\begin{equation}\label{eq:strategy-expansion}
\begin{aligned}
 \frac{\QN^g(w_{t,f,A})}{\CN(\R^n)}
 =1-t^2\bigl(&G_n|\Ric(p)|^2+B_n|W(p)|^2\\
 &+L_n[f]\langle A,\Ric(p)\rangle
 +Q_n[f]|A|^2\bigr)
 +o(t^2),\quad  B_n>0
\end{aligned}
\end{equation}
where $L_n$ and $Q_n$ are, respectively, linear and
quadratic functionals of $f$.
Taking $A=\operatorname{Ric}(p)$ reduces the problem
to the following variational problem
:
find a maximizer $f_*$ of
\[
 \mathcal G_n[f]:=G_n+L_n[f]+Q_n[f]
\]
such that $\mathcal G_n[f_*]>0$; see Section~\ref{sec:variation} for the justification.
Using the coercivity of the associated bilinear form,
we apply the Lax--Milgram theorem to obtain a unique
solution to the Euler--Lagrange equation, which is
the maximizer $f_*$. We then solve the corresponding
radial ordinary differential equation explicitly
in terms of Bessel functions; see \cref{prop:maximizer}.
Finally, we establish the positivity of the optimal value
\[
\Gamma_n:=\mathcal G_n[f_*]>0
\qquad\text{for every }n\ge3;
\]
see \cref{thm:positivity}.

\subsection{Organization of this paper}

Section~\ref{sec:bessel} recalls the Euclidean Nash extremal
and the relevant Bessel identities.
Section~\ref{sec:local-expansions} derives the curvature
expansions of the Nash quotient.
Section~\ref{sec:optimization-positivity} solves the radial
variational problem and establishes positivity of the optimal
Ricci curvature coefficient.
Section~\ref{sec:rigidity} proves the main theorem.
The appendix contains the proofs of the Bessel identities and
radial integral formulas used in the main text.

\section{Preliminaries: The Euclidean extremal function for the Nash inequality, some properties and integral formulas for later use}\label{sec:bessel}
\subsection{The Euclidean extremal function for the Nash inequality}
 
Let $v_n=\omega_{n-1}/n$ denote the volume of the Euclidean unit ball
in $\R^n$, where $\omega_{n-1}$ is the area of the unit sphere.
 Carlen and Loss \cite{CarlenLoss} proved that the sharp constant in Nash's inequality is
 \begin{equation}\label{eq:euclidean-nash-constant}
\CN(\R^n)=
\frac{2^{2/n}n\lambda_1 v_n^{2/n}}{(n+2)^{1+2/n}},
\qquad \lambda_1=j_{n/2,1}^{\,2},
\end{equation}
where $j_{n/2,1}$ denotes the first positive zero of the Bessel function $J_{n/2}$.
Here $\lambda_1$ is the first positive Neumann eigenvalue of $-\Delta$
restricted to radial functions on the Euclidean unit ball in $\R^n$;
it depends only on the dimension $n$.
The functions
 \[
a\,U\!\left(\frac{|y-y_0|}{b}\right),
 \qquad a>0,\quad b>0,\quad y_0\in\R^n
\]
 attain equality in the sharp Euclidean Nash inequality
\eqref{eq:nash-inequality}, where
\begin{equation}\label{eq:profile}
 U(r)=
 \begin{cases}
  1-F(r), & 0\leq r\leq1,\\
  0, & r>1,
 \end{cases}
\end{equation}
and
\begin{equation}\label{eq:F-definition}
 F(r)=
 \begin{cases}
  \dfrac{r^{-\nu}J_\nu(\kappa r)}{J_\nu(\kappa)}, & 0<r\leq1,\\[6pt]
  \dfrac{(\kappa/2)^\nu}{\Gamma(\nu+1)J_\nu(\kappa)}, & r=0,
  \end{cases}.
\end{equation}
Here $J_\beta$ denotes the Bessel function of the first kind of order
$\beta$, and $j_{\beta,1}$ denotes the first positive zero of $J_\beta$.
We set
\begin{equation}\label{eq:bessel-parameters}
 \nu=\frac{n-2}{2},\qquad
 \kappa=j_{\nu+1,1}=j_{n/2,1},\qquad \lambda_1=\kappa^2.
\end{equation}

Recall that for $\beta\geq0$ and $s>0$, the Bessel function of the first kind
is defined by
\begin{equation}\label{eq:bessel-series}
 J_\beta(s)=\sum_{k=0}^{\infty}
 \frac{(-1)^k}{k!\,\Gamma(k+\beta+1)}
 \left(\frac{s}{2}\right)^{2k+\beta},
\end{equation}
where $\Gamma$ denotes the Gamma function.
The function $J_\beta$ satisfies the Bessel equation
\[
 s^2J_\beta''(s)+sJ_\beta'(s)+(s^2-\beta^2)J_\beta(s)=0,
 \qquad s>0.
\]

\subsection{Some properties and integral formulas for later use}
We collect some properties and integral formulas involving Bessel functions 
for use in later calculations. The proofs are given in Appendix \ref{app:bessel} 
for the reader's convenience.

First, we need the following properties of functions $F$ and $U$.
The Bessel equation and the recurrence relations in \eqref{eq:bessel-recurrences} give
\begin{equation}\label{eq:profile-ode}
 F''+\frac{n-1}{r}F'+\lambda_1 F=0,
 \qquad F(1)=1,\quad F'(1)=0.
\end{equation}

\begin{lemma}\label[lemma]{lem:profile}
The function $U(|y|)$ is smooth and strictly positive on $B_1$.
It satisfies $U(1)=U'(1)=0$.
\end{lemma}

We define the following radial integrals. For each nonnegative integer $k$, set
\begin{align}
 M_k&=\int_0^1Fr^{n-1+k}\dd,&
 L_k&=\int_0^1F^2r^{n-1+k}\dd,\label{eq:FL-integrals}\\
 H_{1,k}&=\int_0^1Ur^{n-1+k}\dd,&
 H_{2,k}&=\int_0^1U^2r^{n-1+k}\dd,\label{eq:H-integrals}\\
 I_k&=\int_0^1(U')^2r^{n-1+k}\dd.&&\label{eq:I-integrals}
\end{align}
Since $U=1-F$, one has
\begin{equation}\label{eq:H-from-F}
 H_{1,k}=\frac1{n+k}-M_k,
 \qquad H_{2,k}=\frac1{n+k}-2M_k+L_k.
\end{equation}

We will use the exact values of these integrals for $k=0,2,4$ in later calculations. 
The following lemma gives these values.

\begin{lemma}\label[lemma]{lem:radial-integrals}
The radial integrals satisfy
\begin{equation}\label{eq:radial-integrals-k0}
 M_0=0,\quad L_0=\frac12,\quad
 H_{1,0}=\frac1n,\quad H_{2,0}=\frac{n+2}{2n},\quad I_0=\frac{\lambda_1}2.
\end{equation}
For $k\geq2$,
\begin{align}
 M_k&=\frac{k}{\lambda_1}
       -\frac{k(n+k-2)}{\lambda_1}M_{k-2},\label{eq:M-rec}\\
 I_k&=\lambda_1 L_k-\frac k2+
             \frac{k(n+k-2)}2L_{k-2},\label{eq:I-rec}\\
 L_k&=\frac{\lambda_1-\frac k2(n-2-k)
      +\frac k2(n-2-k)(n+k-2)L_{k-2}}
      {2\lambda_1(k+1)}.\label{eq:L-rec}
\end{align}
In particular,
\begin{align}
 M_2&=\frac2{\lambda_1},\notag\\
 M_4&=\frac4{\lambda_1}-\frac{8(n+2)}{\lambda_1^2},\label{eq:M24}\\
 L_2&=\frac16+\frac{(n-2)(n-4)}{12\lambda_1},\notag\\
 L_4&=\frac1{10}+\frac{(n-6)(n-4)}{30\lambda_1}
 +\frac{(n-6)(n-4)(n-2)(n+2)}{60\lambda_1^2},
 \label{eq:L24}\\
 H_{1,2}&=\frac1{n+2}-\frac2{\lambda_1},\notag\\
 H_{2,2}&=\frac1{n+2}-\frac4{\lambda_1}
 +\frac16+\frac{(n-2)(n-4)}{12\lambda_1},\label{eq:H2}\\
 H_{1,4}&=\frac1{n+4}-\frac4{\lambda_1}+\frac{8(n+2)}{\lambda_1^2},\notag\\
 H_{2,4}&=\frac1{n+4}-\frac8{\lambda_1}+\frac{16(n+2)}{\lambda_1^2}
 +\frac1{10}+\frac{(n-6)(n-4)}{30\lambda_1}\notag\\
 &\qquad{}+\frac{(n-6)(n-4)(n-2)(n+2)}{60\lambda_1^2},
 \label{eq:H4-expanded}\\
 I_2&=\frac{\lambda_1}6+\frac{n^2-4}{12},\notag\\
 I_4&=\frac{\lambda_1}{10}+\frac{n^2-16}{30}
 +\frac{(n^2-4)(n^2-16)}{60\lambda_1}.\label{eq:I24}
\end{align}
\end{lemma}
The proofs of \cref{lem:profile,lem:radial-integrals} are given in
Appendices~\ref{app:bessel} and~\ref{app:radial-integrals}, respectively.

Here and below, $\omega_{n-1}=|\mathbb S^{n-1}|=2\pi^{n/2}/\Gamma(n/2)$ denotes the area of the Euclidean unit sphere.

Notice that the radial integrals in \eqref{eq:radial-integrals-k0} give the 
optimal constant in the Euclidean Nash inequality:
\begin{equation}\label{eq:sharp-constant}
 \CN(\R^n)=
 \frac{(\omega_{n-1}H_{1,0})^{4/n}\omega_{n-1}I_0}
 {(\omega_{n-1}H_{2,0})^{1+2/n}}
 =\frac{2^{2/n}\lambda_1\omega_{n-1}^{2/n}n^{1-2/n}}
 {(n+2)^{1+2/n}}.
\end{equation}

\section{Expansion formulas for the Nash quotient }\label{sec:local-expansions}

\subsection{Vanishing of the scalar curvature}\label{sec:geometry}

By evaluating the Nash quotient on rescaled Euclidean
extremal functions concentrating at a point, Druet,
Hebey, and Vaugon \cite{DruetHebeyVaugon1999}
computed the first nontrivial term in its asymptotic
expansion with respect to the concentration scale
and obtained the following curvature restriction.

\begin{proposition}\label[proposition]{prop:first-order}[Theorem 1.3 in \cite{DruetHebeyVaugon1999}]
Let $(M^n,g)$ be a smooth Riemannian manifold without boundary, with $n\ge2$, and let $V\Subset M$ be a nonempty open set.
If $\CN(V,g)\ge\CN(\R^n)$, then $S\le0$ on $V$.
\end{proposition}

As a direct corollary, we obtain the following result.

\begin{corollary}\label[corollary]{prop:scalar-reduction}
If $\CN(V,g)\ge\CN(\R^n)$ and $\int_VS\,d\mu_g\ge0$, then $S\equiv0$ on $V$.
\end{corollary}

\subsection{The order $t^2$ expansion of the modified Nash quotient}\label{sec:perturbation}
Choose $t_p>0$ so that $B_g(p,\sqrt{t_p})\Subset V$ and
$\sqrt{t_p}<\operatorname{inj}_g(p)$. For $0<t<t_p$, define
$U_{p,t}$ in normal coordinates $\{x_i\}$ centered at $p$ by
\begin{equation}\label{eq:localized-U}
 U_{p,t}(x)=
 \begin{cases}
  t^{-n/4}U\!\left(\dfrac{|x|}{\sqrt t}\right),
  & |x|<\sqrt t,\\[6pt]
  0, & \text{otherwise},
 \end{cases}
\end{equation}
where $U$ is the function defined in \eqref{eq:profile}. Note that $U_{p,t}(x)$ are the extremal functions of Nash quotient on $\mathbb{R}^n$. 

As mentioned in Section~\ref{sec:proof_strategy},
since the expansion of $\QN^g(U_{p,t})$ does not by itself imply
the desired rigidity result, we need to consider the following
modified version. For the
normal coordinates $\{x_i\}$ centered at $p$, define
\begin{equation}\label{eq:perturbed-test}
 w_{t,f,A}(x)=
 \begin{cases}
  t^{-n/4}\left[
   U\!\left(\frac{|x|}{\sqrt t}\right)
   +tf\!\left(\frac{|x|}{\sqrt t}\right)
    z_A\!\left(\frac{x}{|x|}\right)\right],
  & |x|<\sqrt t,\\[6pt]
  0, & \text{otherwise},
 \end{cases}
\end{equation}
where
 $f\in C^\infty([0,1])$ satisfies
\begin{equation}\label{eq:amplitude-class}
 f(r)=r^2\widehat f(r^2)\ \text{near }0,
 \qquad f=0\ \text{on }[1-\delta_f,1],
\end{equation}
for some smooth $\widehat f\in C^\infty([0,1])$ and some $\delta_f>0$, and let $A\in\operatorname{Sym}^2(T_p^*M)$ be a trace-free symmetric tensor, where
\[
z_A(\theta)=A_{ij}\theta^i\theta^j \text{ for } \theta=\frac{x}{|x|}.
\]
Here, $f$ and $A$ will be determined later to achieve our goal.

We derive the expansion of $\QN^g(w_{t,f,A})$ in the following proposition.

\begin{proposition}\label[proposition]{prop:corrected}
Let $f\in C^\infty([0,1])$ satisfy \eqref{eq:amplitude-class}, and
let $A\in\operatorname{Sym}^2(T_p^*M)$ be a trace-free symmetric tensor.
Suppose that $S\equiv 0$ on $V$.
For fixed $p,f,A$, the following expansion holds as $t\downarrow0$
\begin{equation}\label{eq:corrected-quotient}
 \frac{\QN^g(w_{t,f,A})}{\CN(\R^n)}
 =1-t^2\bigl(G_n|\Ric|^2+B_n|W|^2
       +L_n[f]\langle A,\Ric\rangle+Q_n[f]|A|^2\bigr)+o(t^2),
\end{equation}
where $W$ denotes the Weyl curvature tensor. Here,

\begin{equation}\label{eq:G-B}
 G_n:=\frac{(7-2n)P(n,\lambda_1)}
 {270(n-2)\lambda_1^2n(n+2)(n+4)},\qquad
 B_n:=\frac{P(n,\lambda_1)}{360\lambda_1^2n(n+2)(n+4)}.
\end{equation}
\begin{align}
 Q_n[f]&:=-\frac{2\rho_n}{\lambda_1}
 \int_0^1\left((f')^2+\frac{2n}{r^2}f^2-\lambda_1 f^2\right)
 r^{n-1}\dd,
                         \label{eq:Q-form}\\
 L_n[f]&:=\frac{4\rho_n}{3\lambda_1}\int_0^1F'fr^n\dd.
                         \label{eq:L-form}
\end{align}
Put
\begin{equation}\label{eq:P}
 P(n,\lambda_1):=6\lambda_1^2+2(n^2-16)\lambda_1+(n^2-4)(n^2-16),
\end{equation} and $\rho_n:=2/[n(n+2)]$.
Moreover, $B_n>0$ for every $n\ge3$.
\end{proposition}
\begin{proof}
Put $J(x)=\sqrt{\det(g_{ij}(x))}$ and write
\[
 \begin{gathered}
 \Phi(y)=f(r)z_A(\theta)=f(r)A_{ij}\theta^i\theta^j,
 \qquad J_t(y)=J(\sqrt t\,y),\\
 r=|y|,\quad \theta=\frac{y}{|y|}\quad(y\ne0).
 \end{gathered}
\]
In the calculations below, we will use the normal-coordinate volume expansion
\begin{equation}\label{eq:volume-expansion}
 \begin{split}
 J(x)={}&1-\frac16\Ric_{ij}x^ix^j
       -\frac1{12}\nabla_k\Ric_{ij}x^ix^jx^k\\
       &+v_{ijkl}x^ix^jx^kx^l+O(|x|^5).
 \end{split}
\end{equation}
Here $v=(v_{ijkl})$ is the fourth-order coefficient tensor.
All curvature tensors and their covariant derivatives are evaluated
at $p$; see \cite[Lemma~3.4, p.~210]{SchoenYau}
and \cite{ChengGN}.
For a fourth-order tensor $T$, define
\begin{equation}\label{eq:fourth-order-contraction}
 \E(T)=\sum_{i,j}(T_{iijj}+T_{ijij}+T_{ijji}).
\end{equation}
Since $S=0$, spherical integration gives
\begin{equation}\label{eq:volume-average-second-order}
 \frac1{\omega_{n-1}}\int_{\mathbb S^{n-1}}J(\sqrt t\,r\theta)\,d\theta
 =1+\frac{\E(v)}{n(n+2)}t^2r^4+O(t^{5/2}r^5),
 \qquad 0\leq r\leq1.
\end{equation}
The quadratic term vanishes after spherical integration since $S=0$,
and the cubic term vanishes by symmetry. The fourth-order term uses
$\frac1{\omega_{n-1}}\int_{\mathbb S^{n-1}}
v_{ijkl}\theta^i\theta^j\theta^k\theta^l\,d\theta
=\E(v)/[n(n+2)]$;
see \cite[Eqs.~(2.9)--(2.12)]{ChengLS}.

First, we compute
\begin{align}
 a(t)
 &:=\frac{\|w_{t,f,A}\|_2^2}{\omega_{n-1}}\notag\\
 &=\frac1{\omega_{n-1}}\int_{|x|<\sqrt t}
       w_{t,f,A}(x)^2\sqrt{\det g(x)}\,dx\notag\\
 &\overset{x=\sqrt t\,y}{=}
   \frac{t^{-n/2}t^{n/2}}{\omega_{n-1}}
       \int_{B_1}(U+t\Phi)^2J_t\,dy\notag\\
 &\overset{y=r\theta}{=}
   \frac1{\omega_{n-1}}\int_0^1\biggl[
       U^2\int_{\mathbb S^{n-1}}J_t\,d\theta
       +2tUf\int_{\mathbb S^{n-1}}A_{ij}\theta^i\theta^jJ_t\,d\theta\notag\\
 &\hspace{30mm}
       +t^2f^2\int_{\mathbb S^{n-1}}
          (A_{ij}\theta^i\theta^j)^2J_t\,d\theta
       \biggr]r^{n-1}\dd\notag\\
 &=\int_0^1U^2r^{n-1}\dd\notag\\
 &\quad+t^2\int_0^1
       \biggl[\frac{\E(v)}{n(n+2)}r^4U^2
       +\rho_n|A|^2f^2
       -\frac{\rho_n}{3}r^2\langle A,\Ric\rangle Uf
       \biggr]r^{n-1}\dd+o(t^2)\notag\\
 &=H_{2,0}+t^2\biggl\{
       \frac{\E(v)}{n(n+2)}H_{2,4}
       +\rho_n|A|^2\int_0^1f^2r^{n-1}\dd\notag\\
 &\hspace{35mm}
       -\frac{\rho_n}{3}\langle A,\Ric\rangle
           \int_0^1Ufr^{n+1}\dd
       \biggr\}+o(t^2).
 \label{eq:corrected-a}
\end{align}
Here we used the following identities for trace-free symmetric tensors $A$ and $B$:
\[
 \begin{aligned}
 \frac1{\omega_{n-1}}\int_{\mathbb S^{n-1}}
       A_{ij}\theta^i\theta^j\,d\theta&=0,\\
 \frac1{\omega_{n-1}}\int_{\mathbb S^{n-1}}
       (A_{ij}\theta^i\theta^j)(B_{kl}\theta^k\theta^l)\,d\theta
       &=\rho_n\langle A,B\rangle.
 \end{aligned}
\] These identities follow from rotational symmetry.
Taking $B=\Ric$, together with
\eqref{eq:volume-expansion}, \eqref{eq:volume-average-second-order},
and $S=0$, we obtain the spherical integrals in \eqref{eq:corrected-a}.
The last equality uses the definitions of $H_{2,0}$ and $H_{2,4}$
in \eqref{eq:H-integrals}.

Next, we compute the $L^1$ integral. Since $f$ vanishes near $r=1$
and $U>0$ on $[0,1)$, we have $U+tfA_{ij}\theta^i\theta^j\geq0$ for fixed $f,A$
and all sufficiently small $t>0$. Thus the absolute value can be omitted,
and we obtain
\begin{align}
 b(t)
 &:=\frac{t^{-n/4}\|w_{t,f,A}\|_1}{\omega_{n-1}}\notag\\
 &=\frac{t^{-n/4}}{\omega_{n-1}}\int_{|x|<\sqrt t}
       w_{t,f,A}(x)\sqrt{\det g(x)}\,dx\notag\\
 &\overset{x=\sqrt t\,y}{=}
   \frac{t^{-n/4}t^{-n/4}t^{n/2}}{\omega_{n-1}}
       \int_{B_1}(U+t\Phi)J_t\,dy\notag\\
 &\overset{y=r\theta}{=}
   \frac1{\omega_{n-1}}\int_0^1\left[
       U\int_{\mathbb S^{n-1}}J_t\,d\theta
       +tf\int_{\mathbb S^{n-1}}A_{ij}\theta^i\theta^jJ_t\,d\theta
       \right]r^{n-1}\dd\notag\\
 &=\int_0^1Ur^{n-1}\dd
       +t^2\int_0^1\left[
          \frac{\E(v)}{n(n+2)}r^4U
          -\frac{\rho_n}{6}r^2\langle A,\Ric\rangle f
          \right]r^{n-1}\dd+o(t^2)\notag\\
 &=H_{1,0}+t^2\left\{
       \frac{\E(v)}{n(n+2)}H_{1,4}
       -\frac{\rho_n}{6}\langle A,\Ric\rangle
           \int_0^1fr^{n+1}\dd
       \right\}+o(t^2),
 \label{eq:corrected-b}
\end{align}
where the spherical integrations again use
\eqref{eq:volume-expansion} and rotational symmetry.
The last equality substitutes the definitions of $H_{1,0}$
and $H_{1,4}$ from \eqref{eq:H-integrals}.

Finally, we compute the gradient energy
\begin{align}
 d(t)
 &:=\frac{t\|\nabla w_{t,f,A}\|_2^2}{\omega_{n-1}}\notag\\
 &\overset{x=\sqrt t\,y}{=}
   \frac1{\omega_{n-1}}\int_{B_1}
       \biggl[(U')^2+2tU'f'A_{ij}\theta^i\theta^j\notag\\
 &\hspace{22mm}
       +t^2\biggl((f')^2(A_{ij}\theta^i\theta^j)^2
                  +\frac{f^2}{r^2}
                     |\nabla_S(A_{ij}\theta^i\theta^j)|^2\biggr)
       \biggr]J_t\,dy+O(t^3)\notag\\
 &\overset{y=r\theta}{=}
   \frac1{\omega_{n-1}}\int_0^1\biggl[
       (U')^2\int_{\mathbb S^{n-1}}J_t\,d\theta
       +2tU'f'\int_{\mathbb S^{n-1}}
          A_{ij}\theta^i\theta^jJ_t\,d\theta\notag\\
 &\hspace{24mm}
       +t^2(f')^2\int_{\mathbb S^{n-1}}
          (A_{ij}\theta^i\theta^j)^2\,d\theta\notag\\
 &\hspace{24mm}
       +\frac{t^2f^2}{r^2}\int_{\mathbb S^{n-1}}
          |\nabla_S(A_{ij}\theta^i\theta^j)|^2\,d\theta
       \biggr]r^{n-1}\dd+O(t^3)\notag\\
 &=\int_0^1(U')^2r^{n-1}\dd\notag\\
 &\quad+t^2\int_0^1\biggl[
       \frac{\E(v)}{n(n+2)}r^4(U')^2
       -\frac{\rho_n}{3}r^2\langle A,\Ric\rangle U'f'
       \biggr]r^{n-1}\dd\notag\\
 &\quad+t^2\rho_n|A|^2\int_0^1
       \left((f')^2+\frac{2n}{r^2}f^2\right)r^{n-1}\dd
       +o(t^2)\notag\\
 &=I_0+t^2\biggl\{
       \frac{\E(v)}{n(n+2)}I_4
       +\rho_n|A|^2\int_0^1
          \left((f')^2+\frac{2n}{r^2}f^2\right)r^{n-1}\dd
       \notag\\
 &\hspace{25mm}
       -\frac{\rho_n}{3}\langle A,\Ric\rangle
           \int_0^1U'f'r^{n+1}\dd
       \biggr\}+o(t^2).
 \label{eq:corrected-d}
\end{align}
Here we used the following identity to evaluate the angular
gradient term in \eqref{eq:corrected-d}:
\[
 \begin{aligned}
 \frac1{\omega_{n-1}}\int_{\mathbb S^{n-1}}
       |\nabla_S(A_{ij}\theta^i\theta^j)|^2\,d\theta
 &=\frac{2n}{\omega_{n-1}}\int_{\mathbb S^{n-1}}
       (A_{ij}\theta^i\theta^j)^2\,d\theta\\
 &=2n\rho_n|A|^2.
 \end{aligned}
\]
The first equality follows from integration by parts on
$\mathbb S^{n-1}$ and
$-\Delta_S(A_{ij}\theta^i\theta^j)=2nA_{ij}\theta^i\theta^j$;
the second follows
from rotational symmetry.
The remaining spherical integrations use
\eqref{eq:volume-expansion} and rotational symmetry.
The last equality in \eqref{eq:corrected-d} uses
the definitions of $I_0$ and $I_4$
from \eqref{eq:I-integrals}.

Let $a_2,b_2,d_2$ denote the coefficients of $t^2$ in the expansions
of $a(t),b(t),d(t)$, respectively.
The normalized Nash quotient then satisfies
\begin{align*}
 &\quad\frac{\QN^g(w_{t,f,A})}{\CN(\R^n)}
 =\frac{\|w_{t,f,A}\|_1^{4/n}\|\nabla w_{t,f,A}\|_2^2}
         {\CN(\R^n)\|w_{t,f,A}\|_2^{2+4/n}}\\
 &=\left(1+\frac{a_2}{H_{2,0}}t^2+o(t^2)\right)^{-(1+2/n)}
   \left(1+\frac{b_2}{H_{1,0}}t^2+o(t^2)\right)^{4/n}\\
 &\quad\times\left(1+\frac{d_2}{I_0}t^2+o(t^2)\right)\\
 &=1+t^2\left[
       -\left(1+\frac2n\right)\frac{a_2}{H_{2,0}}
       +\frac4n\frac{b_2}{H_{1,0}}+\frac{d_2}{I_0}
       \right]+o(t^2)\\
 &=1-t^2\Biggl[
  \frac{\E(v)}{n(n+2)}
   \left(2H_{2,4}-4H_{1,4}-\frac2{\lambda_1} I_4\right)\\
 &\qquad+\rho_n|A|^2\biggl\{
  2\int_0^1f^2r^{n-1}\,dr
  -\frac2{\lambda_1}\int_0^1
   \left((f')^2+\frac{2n}{r^2}f^2\right)r^{n-1}\,dr
  \biggr\}\\
 &\qquad+\rho_n\langle A,\Ric\rangle\biggl\{
  -\frac23\int_0^1Ufr^{n+1}\,dr
  +\frac23\int_0^1fr^{n+1}\,dr\\
 &\hspace{40mm}
  +\frac2{3\lambda_1}\int_0^1U'f'r^{n+1}\,dr
  \biggr\}\Biggr]+o(t^2)\\
 &=1-t^2\biggl[
  G_n|\Ric|^2+B_n|W|^2+L_n[f]\langle A,\Ric\rangle+Q_n[f]|A|^2
   \biggr]+o(t^2).
\end{align*}

Here we used
\begin{align*}
 &\frac{\E(v)}{n(n+2)}
       \left(2H_{2,4}-4H_{1,4}-\frac2{\lambda_1} I_4\right)\\
 &=\frac{2\E(v)}{n(n+2)}
       \left(L_4-\frac1{n+4}-\frac{I_4}{\lambda_1}\right)\\
 &=-\frac{P(n,\lambda_1)\E(v)}{3\lambda_1^2n(n+2)(n+4)}\\
 &=-\frac{P(n,\lambda_1)}{3\lambda_1^2n(n+2)(n+4)}
       \left(\frac{|\Ric|^2}{45}-\frac{|\Rm|^2}{120}\right)\\
 &=G_n|\Ric|^2+B_n|W|^2.
\end{align*}
To obtain the expansion of the Nash quotient above, we used the sharp constant
in \eqref{eq:sharp-constant}, the values in \eqref{eq:radial-integrals-k0},
and the norm expansions \eqref{eq:corrected-a}--\eqref{eq:corrected-d},
and then applied Taylor's formula.
The components of $v$ in
\eqref{eq:volume-expansion} are
\begin{equation}\label{eq:volume-tensor}
 v_{ijkl}=-\frac1{40}\nabla_k\nabla_l\Ric_{ij}
 -\frac1{180}R_{isjt}R_{kslt}+\frac1{72}\Ric_{ij}\Ric_{kl};
\end{equation}
see \cite[Lemma~3.4, p.~210]{SchoenYau} and \cite{ChengGN}.
Contracting \eqref{eq:volume-tensor} using
\eqref{eq:fourth-order-contraction} gives
\begin{equation}\label{eq:volume-contraction}
 \begin{aligned}
 \E(v)
 &=-\frac1{20}\Delta S+\frac1{72}S^2
   +\frac1{45}|\Ric|^2-\frac1{120}|\Rm|^2\\
 &=\frac1{45}|\Ric|^2-\frac1{120}|\Rm|^2.
 \end{aligned}
\end{equation}
The last equality uses $S\equiv0$ near $p$, which also gives
$\Delta S(p)=0$. The scalar-flat curvature decomposition is
\begin{equation}\label{eq:weyl-decomposition}
 |\Rm|^2=|W|^2+\frac4{n-2}|\Ric|^2,
 \qquad n\ge3.
\end{equation}
In the calculation above, the second equality substitutes $L_4$ and $I_4$ from
\eqref{eq:L24} and \eqref{eq:I24} and uses \eqref{eq:P}.
The third uses \eqref{eq:volume-contraction}, and the last uses
\eqref{eq:weyl-decomposition} and \eqref{eq:G-B}.
This proves \eqref{eq:corrected-quotient}.

It remains to prove $B_n>0$.
The denominator in \eqref{eq:G-B} is positive, so it suffices
to show $P(n,\lambda_1)>0$.
For $n\ge5$, every term in \eqref{eq:P} is positive.
For $n=4$, $P(4,\lambda_1)=6\lambda_1^2>0$.
For $n=3$, the half-integer Bessel formula gives
\[
 J_{3/2}(s)=\sqrt{\frac{2}{\pi s^3}}(\sin s-s\cos s),
 \qquad s>0.
\]
Since $(\sin s-s\cos s)'=s\sin s>0$ on $(0,\pi)$,
this function has no positive zero at or below $\pi$.
Thus $\lambda_1=j_{3/2,1}^2>\pi^2>4$, and
\begin{align*}
 P(3,\lambda_1)
 &=6\lambda_1^2-14\lambda_1-35\\
 &=6(\lambda_1-4)^2+34(\lambda_1-4)+5>0.
\end{align*}
Therefore, $B_n>0$ for every integer $n\ge3$.
\end{proof}

\begin{remark}\label[remark]{rem:ricci-coefficient}
For $n\ge3$, taking $f=0$ and $A=0$ in \eqref{eq:corrected-quotient} yields
\begin{equation}\label{eq:radial-second}
 \frac{\QN^g(U_{p,t})}{\CN(\R^n)}
 =1-t^2\bigl(G_n|\Ric(p)|^2+B_n|W(p)|^2\bigr)+o(t^2).
\end{equation}
From \eqref{eq:G-B}, we have
\[
 G_n=\frac{4(7-2n)}{3(n-2)}B_n.
\]
Note that $B_n>0$, whereas $G_n<0$ for $n\ge4$. This shows that without the correction terms $f$ and $A$, one cannot conclude that $\Rm_g=0$ from \eqref{rem:ricci-coefficient} and the expression for $\QN^g(U_{p,t})$.

One might also try to vary only the tensor $A$ by taking
$f\equiv1$ in \eqref{eq:perturbed-test}. However, this
choice is not admissible, since it does not satisfy
\eqref{eq:amplitude-class}.
To see why a nearly constant radial profile is insufficient,
choose $f_\varepsilon\in\Dn$ (see the definition in \eqref{eq:radial-form}) such that
$f_\varepsilon=1$ on $[2\varepsilon,1-2\varepsilon]$
and $f_\varepsilon=0$ near $r=0$ and $r=1$,
with transition regions of width $O(\varepsilon)$.
The definitions of $L_n$ and $Q_n$ give
\[
L_n[f_\varepsilon]=O(1),
\qquad
Q_n[f_\varepsilon]\le -\frac{c_n}{\varepsilon}
\]
for all sufficiently small $\varepsilon>0$,
where $c_n>0$ depends only on $n$.
Optimizing the quadratic expression in $A$, we obtain
the maximal Ricci curvature coefficient
\[
G_n-\frac{L_n[f_\varepsilon]^2}
{4Q_n[f_\varepsilon]}
=G_n+O(\varepsilon).
\]
Since $G_n<0$ for $n\ge4$, this coefficient remains
negative for sufficiently small $\varepsilon$.
Thus a nearly constant radial profile, even with an
optimal choice of $A$, does not yield the sign needed
for rigidity. This motivates the choice of a suitable
radial function through the variational problem below.

\end{remark}

\section{Optimization of the Ricci coefficient}\label{sec:optimization-positivity}

\subsection{Reduced to a Variational Problem}\label{sec:variation} 

In view of the expansion formula \eqref{eq:corrected-quotient}, 
if we can choose suitable $f$ and $A$ such that for some $C_1>0$
\begin{equation}
 \frac{\QN^g(w_{t,f,A})}{\CN(\R^n)}
 =1-t^2\bigl(C_1|\Ric(p)|^2+B_n|W(p)|^2\bigr)+o(t^2),
\end{equation}
noting that $B_n>0$, we can conclude under the assumptions of Theorem~\ref{thm:main} that $\Ric=0$, $W=0$, and hence $V$ is flat.

Fix a nonzero $f\in\Hn$. The part of the
expansion formula \eqref{eq:corrected-quotient} that depends on $A$ satisfies
\begin{equation}\label{eq:optimal-tensor}
\begin{aligned}
&L_n[f]\langle A,\operatorname{Ric}\rangle
+Q_n[f]|A|^2\\
&\quad=
Q_n[f]\left|
A+\frac{L_n[f]}{2Q_n[f]}\operatorname{Ric}
\right|^2
-\frac{L_n[f]^2}{4Q_n[f]}|\operatorname{Ric}|^2.
\end{aligned}
\end{equation}
We have $Q_n[f]<0$ (see Remark \ref{rem:Q_n}). Consequently, the maximum of \eqref{eq:optimal-tensor} is attained
uniquely at
\begin{equation}\label{eq:optimal-A}
A=-\frac{L_n[f]}{2Q_n[f]}\operatorname{Ric}.
\end{equation}
In other words, $A=c_f\operatorname{Ric}$, where $c_f=-L_n[f]/(2Q_n[f])\in\R$. Since this constant can be absorbed into the choice of $f$, we may simply take $A=\operatorname{Ric}$. In this case, the
expansion formula \eqref{eq:corrected-quotient} becomes
\begin{equation}
 \frac{\QN^g(w_{t,f,A})}{\CN(\R^n)}
 =1-t^2\bigl(\left(G_n+L_n[f]+Q_n[f]\right)|\Ric|^2+B_n|W|^2
       \bigr)+o(t^2), 
\end{equation}
with $B_n>0$.
Hence, to prove Theorem \ref{thm:main} it suffices
to find an admissible radial function $f$ such that
\[
G_n+L_n[f]+Q_n[f]>0.
\]

Let $\Dn$ be the vector space of functions $f\in C^\infty([0,1])$
satisfying \eqref{eq:amplitude-class}. For $f,g\in\Dn$, define the
bilinear form
\begin{equation}\label{eq:radial-form}
 D_n[f,g]=\int_0^1\left(f'g'+\frac{2n}{r^2}fg\right)r^{n-1}\,dr,
 \qquad D_n[f]=D_n[f,f]\ge 0.
\end{equation}
We equip $\Dn$ with the energy norm
$\|f\|_{\Hn}^2=D_n[f]$ and denote its completion by $\Hn$. Denote
\begin{equation}\label{eq:objective}
 \Gf[f]:=G_n+L_n[f]+Q_n[f].
\end{equation}
Here,
\[
 Q_n[f]
 =2\rho_n\left(\int_0^1f^2r^{n-1}\,dr
              -\frac1{\lambda_1} D_n[f]\right)
\]
is the quadratic term and
\[
 L_n[f]
 =\frac{4\rho_n}{3\lambda_1}\int_0^1F'(r)f(r)r^n\,dr
\]
is the linear term.
Then the required positivity is reduced to a variational problem:
establish the existence of a maximizer of $\Gf$ on $\Hn$ and show
that its value is positive.

\subsection{The solution to the variational problem}

We define the bilinear form
\begin{equation}\label{eq:a-form}
 a_n[f,g]=D_n[f,g]-\lambda_1\int_0^1fgr^{n-1}\,dr.
\end{equation}
Hence, $Q_n[f]
 =-\frac{2\rho_n}{\lambda_1}a_n[f,f]$.
We begin with the coercivity estimate that underlies
the existence and uniqueness of the maximizer.

\begin{lemma}[]\label[lemma]{lem:coercivity}
The form $a_n$ is continuous on $\Hn$ and satisfies
\begin{equation}\label{eq:coercivity}
 a_n[f,f]\geq\frac{n+1}{\lambda_1+n+1}D_n[f],
 \qquad f\in\Hn.
\end{equation}
\end{lemma}
\begin{proof}
Set $v=F'$, which is positive on $(0,1)$ by \cref{lem:profile}.
Differentiating \eqref{eq:profile-ode} gives
\[
 -v''-\frac{n-1}{r}v'+\frac{n-1}{r^2}v=\lambda_1 v.
\]
For $f\in\Dn$, write $f=v(f/v)$ and integrate by
parts to obtain
\begin{align}
 a_n[f,f]
 &=\int_0^1v^2\left|\left(\frac f v\right)'\right|^2r^{n-1}\dd
             +(n+1)\int_0^1f^2r^{n-3}\dd
             \label{eq:ground-state}\\
 &\geq(n+1)\int_0^1f^2r^{n-1}\dd\notag\\
 &=\frac{n+1}{\lambda_1}\bigl(D_n[f]-a_n[f,f]\bigr)\notag\\
 &\Longrightarrow\quad
   a_n[f,f]\geq\frac{n+1}{\lambda_1+n+1}D_n[f].\notag
\end{align}
The inequality uses $r^{-2}\geq1$, and the following equality
uses \eqref{eq:a-form}. The final step rearranges the preceding
inequality, proving \eqref{eq:coercivity} for $f\in\Dn$.
Also,
$\int f^2r^{n-1}\dd\leq D_n[f]/(2n)$, so the weighted $L^2$ form
and $a_n$ are continuous in the $D_n$ norm.
\end{proof}

\begin{remark}\label{rem:Q_n}
By Lemma \ref{lem:coercivity} and the fact that $D_n[f]>0$ for $f\not\equiv 0$, we have for $f\not\equiv 0$
\begin{equation*}
Q_n[f]=-\frac{2\rho_n}{\lambda_1}a_n[f,f]
\le -\frac{2\rho_n(n+1)}{\lambda_1(\lambda_1+n+1)}D_n[f]<0.
\end{equation*}
\end{remark}

\begin{remark}\label[remark]{rem:variational-equation}
For $\eta\in C_c^\infty(0,1)\subset\Hn$, we compute the first
variation at a candidate maximizer $f_*$:
\begin{align*}
\left.\frac{d}{ds}\right|_{s=0}\Gf[f_*+s\eta]
&=L_n[\eta]-\frac{4\rho_n}{\lambda_1}a_n[f_*,\eta]\\
&=-\frac{4\rho_n}{\lambda_1}\int_0^1\left[
 f_*'\eta'+\left(\frac{2n}{r^2}-\lambda_1\right)f_*\eta
 -\frac13rF'\eta\right]r^{n-1}\dd\\
&=-\frac{4\rho_n}{\lambda_1}\int_0^1\left[
 -f_*''-\frac{n-1}{r}f_*'
 +\left(\frac{2n}{r^2}-\lambda_1\right)f_*
 -\frac13rF'\right]\eta\,r^{n-1}\dd.
\end{align*}
At a maximizer, the first variation must vanish for every
$\eta\in\Hn$, giving the weak equation \eqref{eq:weak-EL} below.
In particular, the final integral above is zero for every
$\eta\in C_c^\infty(0,1)$. Thus, define the differential operator
\[
 \Ln:=-\frac{d^2}{dr^2}-\frac{n-1}{r}\frac d{dr}
                    +\frac{2n}{r^2}-\lambda_1.
\]
The problem then reduces to solving the following ordinary differential
equation in $\Hn$:
\begin{equation}\label{eq:strong-EL}
 \Ln f_*=\frac13rF'.
\end{equation}
Indeed,  for every nonzero $\eta\in\Hn$,
\[
 \left.\frac{d^2}{ds^2}\right|_{s=0}\Gf[f+s\eta]
 =-\frac{4\rho_n}{\lambda_1}a_n[\eta,\eta]<0.
\]
Hence $\Gf$ is strictly concave, so any stationary point is its
unique global maximizer. The following lemma establishes existence,
gives the corresponding difference identity, and constructs the
solution explicitly.
\end{remark}

\begin{lemma}[]\label[lemma]{prop:maximizer}
There is a unique maximizer $f_*\in\Hn$ of \eqref{eq:objective}.
It is characterized by
\begin{equation}\label{eq:weak-EL}
 a_n[f_*,\eta]=\frac13\int_0^1rF'\eta\,r^{n-1}\dd,
 \qquad \eta\in\Hn,
\end{equation}
and satisfies
\begin{equation}\label{eq:gap}
 \Gf[f_*]-\Gf[f]=\frac{2\rho_n}{\lambda_1}
                      a_n[f-f_*,f-f_*],\qquad f\in\Hn.
\end{equation}
More explicitly \begin{equation}\label{eq:explicit-f}
 f_*(r)=\frac1{12}\bigl(\Psi(r)-r^2F(r)\bigr),
\end{equation}
where
\begin{equation}\label{eq:Psi}
 \Psi(r):=\frac{r^{-\nu}J_{\nu+2}(\kappa r)}{J_{\nu+2}(\kappa)}.
\end{equation}

\end{lemma}
\begin{proof}
To apply the Lax--Milgram theorem, we first verify the continuity
on $\Hn$ of the linear functional on the right-hand side of the weak equation
\eqref{eq:weak-EL}. By Cauchy--Schwarz
and $\int_0^1f^2r^{n-1}\dd\leq D_n[f]/(2n)$,
\begin{align*}
 \left|\int_0^1F'fr^n\dd\right|
 &\leq\left(\int_0^1(F')^2r^{n+1}\dd\right)^{1/2}
       \left(\int_0^1f^2r^{n-1}\dd\right)^{1/2}\\
 &\leq\frac1{\sqrt{2n}}
       \left(\int_0^1(F')^2r^{n+1}\dd\right)^{1/2}
       \|f\|_{\Hn}.
\end{align*}
Thus $L_n$ is continuous on $\Hn$, and the right-hand side of
\eqref{eq:weak-EL} is a continuous linear functional. Together
with \cref{lem:coercivity}, this allows the Lax--Milgram theorem
to yield a unique solution $f_*\in\Hn$ of \eqref{eq:weak-EL}.
Writing \(f=f_*+\eta\) and using
the variational equation \eqref{eq:weak-EL}, we obtain the identity
\eqref{eq:gap}, because
\[
 L_n[\eta]=\frac{4\rho_n}{\lambda_1} a_n[f_*,\eta].
\]
This proves maximality and uniqueness.

To solve the Euler--Lagrange equation \eqref{eq:strong-EL} explicitly,
we first construct a particular solution. A direct calculation gives
\begin{align}
 \Ln(r^2F)
 &=-r^2(F''+\lambda_1 F)-(n+3)rF'\notag\\
 &=(n-1)rF'-(n+3)rF'=-4rF'.\label{eq:particular-solution}
\end{align}
The second equality uses \eqref{eq:profile-ode}. Hence
$-r^2F/12$ is a particular solution of \eqref{eq:strong-EL}.

Next, we solve the homogeneous equation $\Ln h=0$. Set
$h(r)=r^{-\nu}\phi(s)$, where $s=\kappa r$. Substitution gives
\[
 \Ln h=-r^{-\nu-2}
 \left[s^2\phi''+s\phi'+\bigl(s^2-(\nu+2)^2\bigr)\phi\right].
\]
Here derivatives of $\phi$ are taken with respect to $s$, and we used
$\kappa^2=\lambda_1$ and $\nu^2+2n=(\nu+2)^2$.
Thus $\phi$ satisfies the Bessel equation of order $\nu+2$.
Choosing $\phi(s)=J_{\nu+2}(s)$ and returning to the variable $r$
gives a regular homogeneous solution
\[
 h(r)=r^{-\nu}J_{\nu+2}(\kappa r)=O(r^2)
 \qquad(r\downarrow0).
\]
The behaviour at the origin follows from the Bessel series
\eqref{eq:bessel-series}.
The third identity in \eqref{eq:bessel-recurrences} and
$J_{\nu+1}(\kappa)=0$ give
$J_{\nu+2}(\kappa)=-J_\nu(\kappa)\ne0$.
Normalizing this solution to equal one at $r=1$ gives
$\Psi$ in \eqref{eq:Psi}.

Finally, we use $\Psi$ and the particular solution to construct a
solution of \eqref{eq:strong-EL} by setting
\[
 f(r)=-\frac1{12}r^2F(r)+c\Psi(r).
\]
Since functions in $\Hn$ must satisfy $f(1)=0$, we determine $c$ from
\[
 0=f(1)=-\frac1{12}F(1)+c\Psi(1)
       =-\frac1{12}+c
 \quad\Longrightarrow\quad c=\frac1{12}.
\]
Here $F(1)=\Psi(1)=1$. This yields the expression in
\eqref{eq:explicit-f}.
The Bessel series shows that $f(r)=r^2\widehat f(r^2)$ near the
origin for a smooth function $\widehat f$.
Since $f$ is smooth up to $r=1$ and $f(1)=0$, cutting it off near
$r=1$ approximates $f$ in the energy norm, so $f\in\Hn$.
For $\eta\in\Dn$, integration by parts now gives
\[
 a_n[f,\eta]
 =\int_0^1(\Ln f)\eta r^{n-1}\dd
 =\frac13\int_0^1rF'\eta r^{n-1}\dd.
\]
By density and continuity, this identity holds for every
$\eta\in\Hn$. Thus $f$ satisfies \eqref{eq:weak-EL}, and
uniqueness gives $f=f_*$, proving \eqref{eq:explicit-f}.
\end{proof}

\begin{lemma}[]\label[lemma]{prop:optimal-value}
The maximum $\Gamma_n=\Gf[f_*]$ is
\begin{equation}\label{eq:Gamma-gain}
 \Gamma_n=G_n+
 \frac{4(2n+1)\lambda_1+(n+2)(n-4)(n-2)}
 {216n(n+2)\lambda_1^2}.
\end{equation}
Thus, after taking $A=\Ric(p)$ in \eqref{eq:corrected-quotient},
the coefficient of $t^2|\Ric(p)|^2$ is $-\Gf[f]$,
which attains its minimum $-\Gamma_n$ over $\Hn$ at $f=f_*$.
\end{lemma}
\begin{proof}
To compute the maximum, we rewrite $\Psi$ in terms of $F$ and $F'$.
The first and third identities in \eqref{eq:bessel-recurrences} give
\begin{equation}\label{eq:Psi-F}
 \begin{split}
 J_{\nu+2}(\kappa r)
 &=\frac n{\kappa r}J_{\nu+1}(\kappa r)-J_\nu(\kappa r),\\
 \Psi(r)&=F(r)+\frac n{\lambda_1 r}F'(r).
 \end{split}
\end{equation}
We compute
\begin{align*}
 \Gamma_n
 &=G_n+L_n[f_*]+Q_n[f_*]\\
 &=G_n+\frac{4\rho_n}{3\lambda_1}\int_0^1F'f_*r^n\dd
           -\frac{2\rho_n}{\lambda_1}a_n[f_*,f_*]\\
 &=G_n+\frac{\rho_n}{18\lambda_1}
           \int_0^1F'(\Psi-r^2F)r^n\dd\\
 &=G_n+\frac{\rho_n}{18\lambda_1}\biggl\{
       \int_0^1FF'r^n\dd-\int_0^1FF'r^{n+2}\dd\\
 &\hspace{38mm}
       +\frac n{\lambda_1}\int_0^1(F')^2r^{n-1}\dd
       \biggr\}\\
 &=G_n+\frac{\rho_n}{18\lambda_1}
       \left(\frac12-\frac n2L_0-\frac12
                   +\frac{n+2}{2}L_2+\frac n{\lambda_1} I_0\right)\\
 &=G_n+\frac{4(2n+1)\lambda_1+(n+2)(n-4)(n-2)}
                 {216n(n+2)\lambda_1^2}.
\end{align*}
Here we used \eqref{eq:objective}, \eqref{eq:L-form},
and \eqref{eq:weak-EL} with $\eta=f_*$.
The substitutions for $f_*$ and $\Psi$ use
\eqref{eq:explicit-f} and \eqref{eq:Psi-F}.
Integration by parts reduces the remaining integrals to $L_0$, $L_2$,
and $I_0$. The final equality substitutes $L_0=1/2$ and
$I_0=\lambda_1/2$ from \eqref{eq:radial-integrals-k0}, $L_2$ from
\eqref{eq:L24}, and $\rho_n=2/[n(n+2)]$.
This proves \eqref{eq:Gamma-gain}.
\end{proof}

\subsection{Positivity in every dimension}\label{sec:positivity}
Combining \eqref{eq:G-B} and \eqref{eq:Gamma-gain} gives
\begin{equation}\label{eq:Gamma-S}
 \Gamma_n=\frac{S_n(\lambda_1)}
 {1080(n-2)n(n+2)(n+4)\lambda_1^2},
\end{equation}
where
\begin{equation}\label{eq:S}
 \begin{split}
 S_n(s):={}&-24(2n-7)s^2
 +12(n+4)(2n^2+5n-22)s\\
 &+3(6-n)(n^2-4)(n^2-16).
 \end{split}
\end{equation}
We first obtain explicit bounds for \(\lambda_1=j_{n/2,1}^2\), then prove that \(S_n\) is positive throughout the resulting interval.

\begin{lemma}[]\label[lemma]{lem:zero-bounds}
For $n\geq3$ and $\lambda_1=j_{n/2,1}^2$,
\begin{equation}\label{eq:zero-bounds}
 \frac{n^2}{4}<\lambda_1\leq\Lambda_n,
 \qquad \Lambda_n=\left(\frac n2+1\right)
                 \left(\sqrt{\frac n2+2}+1\right)^2.
\end{equation}
\end{lemma}
\begin{proof}
Set $\alpha=n/2$. The function
$y(r)=r^{1/2}J_\alpha(j_{\alpha,1}r)$ solves
\[
 -y''+\frac{\alpha^2-1/4}{r^2}y=\lambda_1 y,
 \qquad y(0)=y(1)=0.
\]
Multiplying the differential equation by \(y\), integrating by parts, and using \(r^{-2}\ge1\) together with the one-dimensional Dirichlet Poincar\'e inequality, we obtain
\[
 \lambda_1=\frac{\int_0^1|y'|^2\dd
       +(\alpha^2-1/4)\int_0^1r^{-2}y^2\dd}{\int_0^1y^2\dd}
 \geq\pi^2+\alpha^2-\frac14>\alpha^2.
\]

For the upper bound we give the Rayleigh--Ritz argument for the
Chambers estimate \cite{Chambers}. The radial Dirichlet operator
$-v''-(2\alpha+1)r^{-1}v'$ has first eigenvalue $j_{\alpha,1}^2$.
Take $v(r)=(1-r^2)^\beta$ with $\beta>1/2$.
Then $v(1)=0$ and
$\int_0^1(v')^2r^{2\alpha+1}\,dr<\infty$.
The Rayleigh--Ritz principle gives
\begin{equation}\label{eq:rayleigh-upper-bound}
\begin{aligned}
 \lambda_1=j_{\alpha,1}^2
 &\leq\frac{\int_0^1(v')^2r^{2\alpha+1}\dd}
      {\int_0^1v^2r^{2\alpha+1}\dd}\\
 &\overset{s=r^2}{=}\frac{4\beta^2\mathrm B(\alpha+2,2\beta-1)}
            {\mathrm B(\alpha+1,2\beta+1)}\\
 &=\frac{2\beta(\alpha+1)(\alpha+2\beta+1)}{2\beta-1}.
\end{aligned}
\end{equation}
Here $\mathrm B$ denotes the Beta function,
\[
 \mathrm B(x,y)=\int_0^1s^{x-1}(1-s)^{y-1}\,ds
 =\frac{\Gamma(x)\Gamma(y)}{\Gamma(x+y)},\qquad x,y>0.
\]
To minimize this upper bound, we use the arithmetic--geometric
mean inequality:
\begin{align*}
 \frac{2\beta(\alpha+1)(\alpha+2\beta+1)}{2\beta-1}
 &=(\alpha+1)\left[\alpha+3+(2\beta-1)
                  +\frac{\alpha+2}{2\beta-1}\right]\\
 &\geq(\alpha+1)\left(\alpha+3+2\sqrt{\alpha+2}\right)\\
 &=(\alpha+1)(\sqrt{\alpha+2}+1)^2.
\end{align*}
Equality holds if and only if $2\beta-1=\sqrt{\alpha+2}$, that is,
\[
 \beta=\frac{1+\sqrt{\alpha+2}}{2}.
\]
Choosing this value of $\beta$ in \eqref{eq:rayleigh-upper-bound} gives
$\lambda_1\leq(\alpha+1)(\sqrt{\alpha+2}+1)^2$.
This proves the upper bound in \eqref{eq:zero-bounds}.
\end{proof}

\begin{theorem}[]\label{thm:positivity}
For every integer $n\geq3$, the maximum
\[
 \Gamma_n=\max_{f\in\Hn}\bigl(G_n+L_n[f]+Q_n[f]\bigr)
         =\Gf[f_*],
\]
given explicitly by \eqref{eq:Gamma-gain}, is strictly positive.
\end{theorem}
\begin{proof}
Suppose first that $n\geq4$. We first compute the two endpoint values
of $S_n$ on $[n^2/4,\Lambda_n]$.
At the lower endpoint in \eqref{eq:zero-bounds}, substituting
$s=n^2/4$ into \eqref{eq:S} gives
\begin{equation}\label{eq:lower-polynomial}
 \begin{split}
 S_n(n^2/4)
 &=\frac{135}{2}n^4+54n^3-624n^2-192n+1152\\
 &\geq S_4(4)=11136>0.
 \end{split}
\end{equation}
The inequality \eqref{eq:lower-polynomial} uses that this polynomial is increasing for $n\geq4$,
since
\[
 \frac{d}{dn}S_n(n^2/4)
 =n(270n^2-1248)+162n^2-192
 \geq3072n+2400>0.
\]
At the upper endpoint $\Lambda_n$ defined in \cref{lem:zero-bounds},
substituting $s=\Lambda_n$ into \eqref{eq:S} and simplifying gives
\begin{equation}\label{eq:upper-polynomial}
 \begin{split}
 S_n(\Lambda_n)={}&\frac{87}{2}n^4+234n^3+336n^2+504n+840\\
 &+24(n+2)(2n^2+15n-2)\sqrt{\frac n2+2}>0.
 \end{split}
\end{equation}
All terms in \eqref{eq:upper-polynomial} are positive for $n\geq4$,
since $2n^2+15n-2\geq90$.
By \eqref{eq:S}, $S_n(s)$ is a concave quadratic polynomial in $s$,
so its value on this interval is bounded below by the minimum of
its two endpoint values. Since $n^2/4<\lambda_1\leq\Lambda_n$ by
\cref{lem:zero-bounds}, we obtain
\[
 S_n(\lambda_1)\geq
 \min\{S_n(n^2/4),S_n(\Lambda_n)\}>0.
\]

For $n=3$,
\begin{align*}
 S_3(s)&=24s^2+924s-315,\\
 S_3'(s)&=48s+924>0\qquad(s>0).
\end{align*}
The lower bound in \eqref{eq:zero-bounds} gives $\lambda_1>9/4$, and hence
\[
 S_3(\lambda_1)>S_3(9/4)=\frac{3771}{2}>0.
\]
The denominator in
\eqref{eq:Gamma-S} is positive, completing the proof.
\end{proof}
\section{Proof of the rigidity theorem}\label{sec:rigidity}

Finally, we give the proof of Theorem \ref{thm:main}.

\begin{proof}[Proof of Theorem~\ref{thm:main}]
By \cref{prop:scalar-reduction},
\begin{equation}\label{eq:scalar-zero}
 S\equiv0\quad\text{on }V.
\end{equation}
If $n=2$, the identity
$R_{ijkl}=(S/2)(g_{ik}g_{jl}-g_{il}g_{jk})$ already gives $\Rm=0$.

Since $\Gf[f_*]=\Gamma_n>0$ and $\Gf$ is continuous
on $\Hn$, the density of $\Dn$ in $\Hn$ allows us
to choose $f_{\varepsilon}\in\Dn$ such that
\[
\Gf[f_{\varepsilon}]>0.
\]
Since $f_{\varepsilon}$ vanishes near $r=1$ and $U>0$ on
$[0,1)$, the ratio $f_{\varepsilon}/U$ is bounded.
Consequently, for any fixed trace-free symmetric
tensor $A$, we have
\[
U(r)+tf_{\varepsilon}(r)A_{ij}\theta^i\theta^j\ge0
\]
for all sufficiently small $t>0$.
Thus $w_{t,f_{\varepsilon},A}$ is an admissible test function. Applying the expansion in \cref{prop:corrected}, we obtain
\begin{align}
 1
 &\leq\frac{\CN(V,g)}{\CN(\R^n)}
 \leq\frac{\QN^g(w_{t,f_\varepsilon,\Ric(p)})}{\CN(\R^n)}
 \notag\\
 &=1-t^2\bigl(\Gf[f_\varepsilon]|\Ric(p)|^2
                      +B_n|W(p)|^2\bigr)+o(t^2)
  \notag\\
  &\Longrightarrow\quad
  \Gf[f_\varepsilon]|\Ric(p)|^2+B_n|W(p)|^2\leq o(1).
 \label{eq:rigidity-quotient}
\end{align}
Letting $t\to0$ in \eqref{eq:rigidity-quotient}, we obtain
\[
 \Gf[f_\varepsilon]|\Ric(p)|^2+B_n|W(p)|^2\leq0.
\]
 $B_n>0$ by \cref{prop:corrected}.
Hence $\Ric(p)=W(p)=0$, and \eqref{eq:weyl-decomposition}
implies $\Rm(p)=0$. Since $p$ was arbitrary, the conclusion follows.
\end{proof}

\appendix

\section{Appendix}\label[appendix]{app:auxiliary}

\subsection{Bessel functions and recurrence identities}\label[subsection]{app:bessel}

For $\beta\geq0$ and $s>0$, the Bessel function of the first kind
is defined by
\begin{equation}\label{eq:bessel-series-appendix}
 J_\beta(s)=\sum_{k=0}^{\infty}
 \frac{(-1)^k}{k!\,\Gamma(k+\beta+1)}
 \left(\frac{s}{2}\right)^{2k+\beta},
\end{equation}
where $\Gamma$ denotes the Gamma function.
The series and its derivatives converge uniformly on compact
subsets of $(0,\infty)$, so termwise differentiation is valid.
The recurrence identities used in the main text are
\begin{equation}\label{eq:bessel-recurrences}
 \begin{aligned}
  J_\beta'(s)&=\frac{\beta}{s}J_\beta(s)-J_{\beta+1}(s),\\
  J_{\beta+1}'(s)&=J_\beta(s)-\frac{\beta+1}{s}J_{\beta+1}(s),\\
  J_{\beta+2}(s)&=\frac{2(\beta+1)}{s}J_{\beta+1}(s)-J_\beta(s).
 \end{aligned}
\end{equation}

Termwise differentiation gives
\begin{align*}
 J_\beta'(s)
 &=\sum_{k=0}^{\infty}
   \frac{(-1)^k(2k+\beta)s^{2k+\beta-1}}
        {2^{2k+\beta}k!\,\Gamma(k+\beta+1)}\\
 &=\frac{\beta}{s}J_\beta(s)
   +\sum_{k=1}^{\infty}
   \frac{(-1)^k\,2k\,s^{2k+\beta-1}}
        {2^{2k+\beta}k!\,\Gamma(k+\beta+1)}\\
 &\overset{k=j+1}{=}\frac{\beta}{s}J_\beta(s)
   -\sum_{j=0}^{\infty}
   \frac{(-1)^j s^{2j+\beta+1}}
        {2^{2j+\beta+1}j!\,\Gamma(j+\beta+2)}\\
 &=\frac{\beta}{s}J_\beta(s)-J_{\beta+1}(s).
\end{align*}
Here we used \eqref{eq:bessel-series-appendix} and $k/k!=1/(k-1)!$.
Likewise,
\begin{align*}
 J_{\beta+1}'(s)
 &=\sum_{k=0}^{\infty}
   \frac{(-1)^k(2k+\beta+1)s^{2k+\beta}}
        {2^{2k+\beta+1}k!\,\Gamma(k+\beta+2)}\\
 &=-\frac{\beta+1}{s}J_{\beta+1}(s)
   +\sum_{k=0}^{\infty}
   \frac{(-1)^k\,2(k+\beta+1)s^{2k+\beta}}
        {2^{2k+\beta+1}k!\,\Gamma(k+\beta+2)}\\
 &=-\frac{\beta+1}{s}J_{\beta+1}(s)
   +\sum_{k=0}^{\infty}
   \frac{(-1)^k s^{2k+\beta}}
        {2^{2k+\beta}k!\,\Gamma(k+\beta+1)}\\
 &=J_\beta(s)-\frac{\beta+1}{s}J_{\beta+1}(s).
\end{align*}
The simplification uses
$\Gamma(k+\beta+2)=(k+\beta+1)\Gamma(k+\beta+1)$,
together with \eqref{eq:bessel-series-appendix}.
Combining these identities, we obtain
\begin{align*}
 J_{\beta+2}(s)
 &=\frac{\beta+1}{s}J_{\beta+1}(s)-J_{\beta+1}'(s)\\
 &=\frac{\beta+1}{s}J_{\beta+1}(s)
   -\left(J_\beta(s)-\frac{\beta+1}{s}J_{\beta+1}(s)\right)\\
 &=\frac{2(\beta+1)}{s}J_{\beta+1}(s)-J_\beta(s).
\end{align*}
We used the first identity in \eqref{eq:bessel-recurrences}
at order $\beta+1$ and then the second.

Finally,
\begin{align}
 J_\beta''(s)
 &=\frac{\beta}{s}J_\beta'(s)
   -\frac{\beta}{s^2}J_\beta(s)-J_{\beta+1}'(s)\notag\\
 &=\frac{\beta}{s}J_\beta'(s)
   -\frac{\beta}{s^2}J_\beta(s)-J_\beta(s)
   +\frac{\beta+1}{s}J_{\beta+1}(s)\notag\\
 &=\frac{\beta}{s}J_\beta'(s)
   -\frac{\beta}{s^2}J_\beta(s)-J_\beta(s)
   +\frac{\beta+1}{s}
       \left(\frac{\beta}{s}J_\beta(s)-J_\beta'(s)\right)\notag\\
 &=-\frac1sJ_\beta'(s)
   +\left(\frac{\beta^2}{s^2}-1\right)J_\beta(s)\notag\\
 &\Longrightarrow\quad
 s^2J_\beta''(s)+sJ_\beta'(s)+(s^2-\beta^2)J_\beta(s)=0.
 \label{eq:bessel-equation}
\end{align}
The first equality differentiates the first identity in
\eqref{eq:bessel-recurrences}; the substitutions use its first
two identities. The final line is the Bessel equation.

\begin{proof}[Proof of Lemma~\ref{lem:profile}]
The series expansion of $J_\nu$ shows that
$r^{-\nu}J_\nu(\kappa r)$ has a convergent power series in $r^2$
near zero. Hence it is a smooth function of $r^2$ there. The zero $\kappa$ is simple. Indeed, if
$J_{\nu+1}(\kappa)=J_{\nu+1}'(\kappa)=0$,
uniqueness for the Bessel equation at $\kappa>0$ would imply
$J_{\nu+1}\equiv0$, a contradiction.
Since $J_{\nu+1}>0$ on $(0,\kappa)$, it follows that
$J_{\nu+1}'(\kappa)<0$. The second identity in \eqref{eq:bessel-recurrences} gives
$J_\nu(\kappa)=J_{\nu+1}^{\prime}(\kappa)<0$.
Using the first identity in \eqref{eq:bessel-recurrences}, we obtain
\[
 F'(r)=-\frac{\kappa r^{-\nu}J_{\nu+1}(\kappa r)}{J_\nu(\kappa)}>0.
\]
Since $F(1)=1$, this proves positivity of $U$.
\end{proof}

\subsection{Radial integral identities}\label[subsection]{app:radial-integrals}

\begin{proof}[Proof of Lemma~\ref{lem:radial-integrals}]
For $M_0$,
\begin{align*}
 M_0
 &=\int_0^1Fr^{n-1}\dd\\
 &=-\frac1{\lambda_1}\int_0^1(r^{n-1}F')'\dd\\
 &=-\frac1{\lambda_1}\bigl[r^{n-1}F'\bigr]_0^1=0.
\end{align*}
For $k\geq2$,
\begin{align*}
 M_k
 &=\int_0^1Fr^{n-1+k}\dd\\
 &=-\frac1{\lambda_1}\int_0^1r^k(r^{n-1}F')'\dd\\
 &=-\frac1{\lambda_1}\bigl[r^{n+k-1}F'\bigr]_0^1
       +\frac{k}{\lambda_1}\int_0^1F'r^{n+k-2}\dd\\
 &=\frac{k}{\lambda_1}\bigl[r^{n+k-2}F\bigr]_0^1
       -\frac{k(n+k-2)}{\lambda_1}\int_0^1Fr^{n+k-3}\dd\\
 &=\frac{k}{\lambda_1}
       -\frac{k(n+k-2)}{\lambda_1}M_{k-2}.
\end{align*}
These calculations use the divergence form
$(r^{n-1}F')'=-\lambda_1 r^{n-1}F$ of \eqref{eq:profile-ode}
and the definition of $M_k$ in \eqref{eq:FL-integrals}.
Here and in the integrations below, we use $F(1)=1$, $F'(1)=0$,
and $F=O(1)$, $F'=O(r)$ near zero.
In particular, the boundary terms at zero vanish.
This proves the value of $M_0$ and \eqref{eq:M-rec}.

For $k\geq2$, integration by parts gives
\begin{align*}
 I_k
 &=\int_0^1(F')^2r^{n-1+k}\dd\\
 &=\bigl[FF'r^{n-1+k}\bigr]_0^1
       -\int_0^1F\left(F''+\frac{n-1+k}{r}F'\right)
                         r^{n-1+k}\dd\\
 &=\lambda_1 L_k-k\int_0^1FF'r^{n+k-2}\dd\\
 &=\lambda_1 L_k-\frac{k}{2}\bigl[F^2r^{n+k-2}\bigr]_0^1
       +\frac{k(n+k-2)}2\int_0^1F^2r^{n+k-3}\dd\\
 &=\lambda_1 L_k-\frac{k}{2}
       +\frac{k(n+k-2)}2L_{k-2}.
\end{align*}
We used $U'=-F'$ from \eqref{eq:profile},
the differential equation \eqref{eq:profile-ode},
and the definitions \eqref{eq:FL-integrals} and
\eqref{eq:I-integrals}. This proves \eqref{eq:I-rec}.
For $k=0$, the same integration by parts gives
\begin{align*}
 I_0
 &=\int_0^1(F')^2r^{n-1}\dd\\
 &=\bigl[FF'r^{n-1}\bigr]_0^1
       -\int_0^1F(r^{n-1}F')'\dd\\
 &=\lambda_1\int_0^1F^2r^{n-1}\dd
 =\lambda_1 L_0.
\end{align*}

For every integer $k\geq0$, we also have
\begin{align}
 \lambda_1\bigl((n+k)L_k-1\bigr)
 &=-\lambda_1\int_0^1(F^2)'r^{n+k}\dd\notag\\
 &=-2\lambda_1\int_0^1FF'r^{n+k}\dd\notag\\
 &=2\int_0^1F''F'r^{n+k}\dd+2(n-1)I_k\notag\\
 &=\bigl[(F')^2r^{n+k}\bigr]_0^1
       -(n+k)I_k+2(n-1)I_k\notag\\
 &=(n-2-k)I_k.
 \label{eq:pohozaev}
\end{align}
Here the first equality uses $F(1)=1$ and the definition of $L_k$;
the remaining equalities use \eqref{eq:profile-ode} and integration
by parts.

Taking $k=0$ in \eqref{eq:pohozaev} and using $I_0=\lambda_1 L_0$,
\[
 2\lambda_1 L_0
 =n\lambda_1 L_0-(n-2)I_0
 =\lambda_1.
\]
Consequently,
\begin{align*}
 L_0&=\frac12,\qquad I_0=\lambda_1 L_0=\frac{\lambda_1}{2},\\
 H_{1,0}&=\frac1n-M_0=\frac1n,\\
 H_{2,0}&=\frac1n-2M_0+L_0
          =\frac1n+\frac12=\frac{n+2}{2n}.
\end{align*}
The last two calculations use \eqref{eq:H-from-F}.
This proves \eqref{eq:radial-integrals-k0}.

For $k\geq2$,
\begin{align*}
 2\lambda_1(k+1)L_k
 &=\lambda_1(n+k)L_k-\lambda_1(n-2-k)L_k\\
 &=\lambda_1+(n-2-k)(I_k-\lambda_1 L_k)\\
 &=\lambda_1-\frac{k}{2}(n-2-k)
       +\frac{k}{2}(n-2-k)(n+k-2)L_{k-2}.
\end{align*}
The second equality uses \eqref{eq:pohozaev}, and the last uses
\eqref{eq:I-rec}. Dividing by $2\lambda_1(k+1)$ gives
\eqref{eq:L-rec}.

For the values with $k=2,4$, we compute
\begin{align*}
 M_2&=\frac2{\lambda_1}-\frac{2n}{\lambda_1}M_0
       =\frac2{\lambda_1},\\
 M_4&=\frac4{\lambda_1}-\frac{4(n+2)}{\lambda_1}M_2
       =\frac4{\lambda_1}-\frac{8(n+2)}{\lambda_1^2}.
\end{align*}
These are the substitutions $k=2,4$ in \eqref{eq:M-rec}, proving
\eqref{eq:M24}.

For the squared integrals,
\begin{align*}
 L_2
 &=\frac{\lambda_1-(n-4)+n(n-4)L_0}{6\lambda_1}\\
 &=\frac{\lambda_1+(n-2)(n-4)/2}{6\lambda_1}\\
 &=\frac16+\frac{(n-2)(n-4)}{12\lambda_1},\\
 L_4
 &=\frac{\lambda_1-2(n-6)+2(n-6)(n+2)L_2}{10\lambda_1}\\
 &=\frac1{10}+\frac{n-6}{5\lambda_1}
       \left[\frac{n-4}{6}
         +\frac{(n+2)(n-2)(n-4)}{12\lambda_1}\right]\\
 &=\frac1{10}+\frac{(n-6)(n-4)}{30\lambda_1}
       +\frac{(n-6)(n-4)(n-2)(n+2)}{60\lambda_1^2}.
\end{align*}
Here we used \eqref{eq:L-rec}, first with $k=2$ and $L_0=1/2$,
and then with $k=4$ and the computed value of $L_2$.
This proves \eqref{eq:L24}.

For the energy integrals,
\begin{align*}
 I_2
 &=\lambda_1 L_2-1+nL_0\\
 &=\frac{\lambda_1}{6}+\frac{(n-2)(n-4)}{12}
                         -1+\frac n2\\
 &=\frac{\lambda_1}{6}+\frac{n^2-4}{12},\\
 I_4
 &=\lambda_1 L_4-2+2(n+2)L_2\\
 &=\frac{\lambda_1}{10}
       +\left[\frac{(n-6)(n-4)}{30}-2+\frac{n+2}{3}\right]\\
 &\quad+\frac{(n^2-4)(n-4)}{60\lambda_1}(n-6+10)\\
 &=\frac{\lambda_1}{10}+\frac{n^2-16}{30}
       +\frac{(n^2-4)(n^2-16)}{60\lambda_1}.
\end{align*}
These calculations use \eqref{eq:I-rec} and the preceding values
of $L_0,L_2,L_4$, giving \eqref{eq:I24}.

Finally,
\begin{align*}
 H_{1,2}
 &=\frac1{n+2}-M_2
   =\frac1{n+2}-\frac2{\lambda_1},\\
 H_{2,2}
 &=\frac1{n+2}-2M_2+L_2\\
 &=\frac1{n+2}-\frac4{\lambda_1}
       +\frac16+\frac{(n-2)(n-4)}{12\lambda_1},\\
 H_{1,4}
 &=\frac1{n+4}-M_4
   =\frac1{n+4}-\frac4{\lambda_1}+\frac{8(n+2)}{\lambda_1^2},\\
 H_{2,4}
 &=\frac1{n+4}-2M_4+L_4\\
 &=\frac1{n+4}-\frac8{\lambda_1}+\frac{16(n+2)}{\lambda_1^2}
       +\frac1{10}\\
 &\quad+\frac{(n-6)(n-4)}{30\lambda_1}
       +\frac{(n-6)(n-4)(n-2)(n+2)}{60\lambda_1^2}.
\end{align*}
We used \eqref{eq:H-from-F} and the values of $M_2,M_4,L_2,L_4$.
This proves \eqref{eq:H2} and \eqref{eq:H4-expanded}.
\end{proof}

\textbf{AI Disclosure:} The research ideas and overall direction of this work were solely conceived by the authors. Specific calculations in Lemma~\ref{prop:optimal-value}, Lemma~\ref{lem:zero-bounds}, Theorem~\ref{thm:positivity}, and the Appendix were performed with the aid of ChatGPT. The authors have thoroughly reviewed all AI-assisted content and take full responsibility for the accuracy and integrity of the entire manuscript.


\begin{thebibliography}{9}

\bibitem{CarlenLoss}
E.~A. Carlen and M. Loss,
\emph{Sharp constant in Nash's inequality},
International Mathematics Research Notices \textbf{1993} (1993), no.~7, 213--215.

\bibitem{CarlenLieb}
E.~A. Carlen and E.~H. Lieb,
\emph{A dual form of the sharp Nash inequality and its weighted generalization},
Bulletin of the London Mathematical Society \textbf{51} (2019), no.~1, 129--144.

\bibitem{Humbert}
E. Humbert,
\emph{Extremal functions for the sharp $L^2$-Nash inequality},
Calculus of Variations and Partial Differential Equations \textbf{22} (2005), no.~1, 21--44.


\bibitem{Chambers}
Ll.~G.~Chambers,
\emph{An upper bound for the first zero of Bessel functions},
Math. Comp. \textbf{38} (1982), no.~158, 589--591.


\bibitem{ChengLS}
L. Cheng,
\emph{The power series expansions of logarithmic Sobolev, $\mathcal{W}$-functionals and scalar curvature rigidity},
arXiv:2409.06117v5 (2025).

\bibitem{ChengGN}
L. Cheng,
\emph{On the rigidity of manifolds with respect to the Gagliardo--Nirenberg inequalities},
Journal de Math\'ematiques Pures et Appliqu\'ees \textbf{215} (2026), Article~103960.




\bibitem{Davies1989}
E.~B. Davies,
\emph{Heat Kernels and Spectral Theory},
Cambridge Tracts in Mathematics, vol.~92,
Cambridge University Press, Cambridge, 1989.

\bibitem{DruetHebeyVaugon1999}
O.~Druet, E.~Hebey and M.~Vaugon,
\emph{Optimal Nash's inequalities on Riemannian
manifolds: The influence of geometry},
International Mathematics Research Notices
 (1999), no.~14, 735--779.

          \bibitem{Gromovlecture}
          Mikhael Gromov, \textit{Four lectures on scalar curvature}, 2019. arXiv:1908.10612.
          
          \bibitem{Gdozen}
          Mikhael Gromov. A dozen problems, \textit{questions and conjectures about positive scalar curvature.} Foundations of Mathematics and Physics One Century After Hilbert: New Perspectives, 2018: 135-158.
          
          

\bibitem{Nash1958}
J.~Nash,
\emph{Continuity of solutions of parabolic and elliptic equations},
American Journal of Mathematics \textbf{80} (1958),
no.~4, 931--954.


\bibitem{SaloffCoste2002}
L.~Saloff-Coste,
\emph{Aspects of Sobolev-Type Inequalities},
London Mathematical Society Lecture Note Series,
vol.~289, Cambridge University Press, Cambridge, 2002.

\bibitem{Brendle}
          
          Brendle, S.
          \textit{Rigidity phenomena involving scalar curvature}. Surveys in differential geometry. Vol. XVII, 179-202.
          Surv. Differ. Geom., 17

\bibitem{SchoenYau}
R. Schoen and S.-T. Yau,
\emph{Lectures on Differential Geometry},
Conference Proceedings and Lecture Notes in Geometry and Topology, vol.~I,
International Press, Cambridge, MA, 1994.

\end{thebibliography}
\end{document}